\documentclass[11pt, leqno]{amsart}

\makeatletter

\def\subsection{\@startsection{subsection}{2}%
	{\parindent}{.5\linespacing}{-.5em}%
	{\normalfont\itshape\/}}
\def\subsubsection{\@startsection{subsubsection}{3}%
	{\parindent}{.5\linespacing}{-.5em}%
	{\normalfont\itshape\/}}
\def\appendix{\par\c@section\z@ \c@subsection\z@
	\gdef\theHsection{\Hy@AlphNoErr{section}}%
\let\sectionname\appendixname{}
\def\thesection{{\upshape\@Alph\c@section}}}

\newtheoremstyle{plain}{0.5\linespacing}{0.5\linespacing}{\itshape}%
	{\parindent}{\scshape}{.}{0.5em}%
	{\thmname{#1}\thmnumber{ #2}\thmnote{\normalfont{} (#3)}}
\newtheoremstyle{definition}{0.5\linespacing}{0.5\linespacing}%
	{\upshape}{\parindent}%
	{\itshape}{.}{0.5em} 
	{\thmname{#1}\thmnumber{ #2}\thmnote{\normalfont{} (#3)}}
\newtheoremstyle{remark}{0.5\linespacing}{0.5\linespacing}%
	{\upshape}{\parindent}%
	{\itshape}{.}{0.5em} 
	{\thmname{#1}\thmnumber{ #2}\thmnote{\normalfont{} (#3)}}

\renewenvironment{proof}[1][\proofname]{\par 
	\pushQED{\qed}%
	\normalfont\topsep6\p@\@plus6\p@\relax
	\trivlist%
	\item[\hskip\labelsep\hskip\parindent%
	\itshape%
	#1\@addpunct{.}]\ignorespaces%
	}{ 
	\popQED\endtrivlist\@endpefalse%
}
\makeatother
\allowdisplaybreaks

\usepackage{amsmath, amsthm, amssymb}
\usepackage{enumitem}
\usepackage{xcolor}
\definecolor{cite}{rgb}{0.30,0.60,1.00}
\definecolor{url}{rgb}{0.00,0.00,0.80}
\definecolor{link}{rgb}{0.40,0.10,0.20}

\usepackage[pdfusetitle,colorlinks,linkcolor=link,urlcolor=url,citecolor=cite,breaklinks,bookmarksdepth=3,bookmarksopen=true]{hyperref}
\usepackage[top=1.75in, headsep=0.25in, left=1.65in, right=1.7in,bottom=1.275in]{geometry}
\usepackage{url}
\usepackage{graphicx}
\usepackage{caption}
\usepackage{subcaption}
\usepackage[capitalise]{cleveref}
\usepackage{mathdots}
\usepackage{tikz-cd}
\usepackage{array}

\usepackage{mathtools}
\usepackage{physics}
\usepackage{float}
\usepackage[mathcal]{eucal}
\usepackage{microtype}
\usepackage{pagesel}
\usepackage{afterpage}
\usepackage{placeins}

\newtheorem{theorem}{Theorem}[section]

\newtheorem{proposition}[theorem]{Proposition}
\newtheorem{lemma}[theorem]{Lemma}

\newtheorem{corollary}[theorem]{Corollary}

\theoremstyle{definition}
\newtheorem{definition}[theorem]{Definition}
\theoremstyle{definition}

\theoremstyle{remark}
\newtheorem{remark}[theorem]{Remark}
\theoremstyle{remark}

\newcommand{\nNatural}{\mathbb{N}}
\newcommand{\zIntegers}{\mathbb{Z}}
\newcommand{\qRationals}{\mathbb{Q}}
\newcommand{\rReal}{\mathbb{R}}
\newcommand{\rRealPos}{\rReal_{+}}
\newcommand{\cComplex}{\mathbb{C}}

\newcommand{\idmap}{\mathrm{id}}
\renewcommand{\abs}[1]{\left|#1\right|}

\newcommand{\pullback}[1]{{#1}^{*}}

\newcommand{\comp}{\mathbin{\circ}}

\let\setminusaux\setminus{}	
\renewcommand*{\setminus}{\scalemath{\mathrel}{0.7}{\setminusaux}}

\newcommand{\st}{\mid}
\newcommand{\restrict}[2]{{#1}|_{#2}}

\DeclareMathOperator{\graph}{graph}

\renewcommand{\exp}{\operatorname{exp}}

\renewcommand{\leq}{\leqslant}

\let\tmp\epsilon{}
\let\epsilon\varepsilon{}
\let\varepsilon\tmp{}

\let\tmp\phi{}
\let\phi\varphi{}
\let\varphi\tmp{}

\let\emptyset\varnothing

\newcommand{\structure}{\mathcal{S}}
\newcommand{\RrPfaff}{\rReal_{\mathrm{rPfaff}}}
\newcommand{\Rexp}{\rReal_{\exp}}
\newcommand{\Ran}{\rReal_{\mathrm{an}}}
\newcommand{\Ralg}{\rReal_{\mathrm{alg}}}

\newcommand{\s}{\raisebox{.5ex}{\scalebox{0.6}{\#}}}
\newcommand{\so}{\s\kern-.02em{}o}
\NewDocumentCommand{\format}{sO{F}}{
	\IfBooleanTF{#1}
		{#2, \ell}
		{#2}
	}
\newcommand{\degree}[1][D]{#1}
\NewDocumentCommand{\FD}{sooo}{
	\IfBooleanTF{#1}								
		{											
			\IfNoValueTF{#4}
			{\mathcal{S}_{\order[#2]{1},\poly_{#2}\!\qty(#3)}}
			{
				{\mathcal{S}^{#4}_{\order[#2]{1},\poly_{#2}\!\qty(#3)}}
			}
		}
		{											
		\IfNoValueTF{#2}
			{\mathcal{S}}
			{
				\IfNoValueTF{#4}
				{
					{\mathcal{S}_{#2,#3}}
				}
				{\mathcal{S}^{#4}_{#2,#3}}
				}		
		}
}
\NewDocumentCommand{\polyfd}{oo}{
	\IfValueTF{#1}
		{\poly_{\format[#1]}\qty(\degree[#2])}
		{\poly_{\format}\qty(\degree)}
}

\newcommand{\an}[1]{{#1}^{\textrm{a.g.}}}

\newcommand{\cell}[1][C]{\mathcal{#1}}
\newcommand{\hyperbolicParameter}[1]{\{#1\}}
\newcommand{\point}{*}
\newcommand{\disc}[1][\relax]{
	\ifx\relax#1 
		D
	\else
		D\qty(#1)
	\fi
}

\NewDocumentCommand{\puncDisc}{so}{
	\IfBooleanTF{#1}
		{D_{\infty}}
		{D_{\circ}}
	\IfNoValueF{#2}{
		\qty(#2)
		}
}

\NewDocumentCommand{\annulus}{oo}{
	\IfNoValueTF{#1} 	
		{A} 			
		{A\qty(#1,#2)} 	
}
\RenewDocumentCommand{\circle}{o}{
	\IfNoValueTF{#1}{S}{S\qty(#1)}	
}

\NewDocumentCommand{\ext}{smm}{
	\IfBooleanTF{#1}		
		{\qty(#2){}^{#3}} 	
		{#2^{#3}}			
}
\NewDocumentCommand{\hExt}{smm}{
	\IfBooleanTF{#1}			
		{\qty(#2){}^{\hyperbolicParameter{#3}}} 
		{#2^{\hyperbolicParameter{#3}}}			
}
\newcommand{\initial}[3]{#1_{#2..#3}}
\NewDocumentCommand{\nuCover}{mo}{
	\IfNoValueTF{#2}
		{{#1}_{\times\nu}}
		{{#1}_{\times#2}}
}

\newcommand{\projbase}{\initial{\pi}{1}{\ell}}

\newcommand{\poly}{\operatorname{poly}}
\newcommand{\polyl}{\poly_{\ell}}
\newcommand{\varX}{\mathbf{x}}
\newcommand{\varZ}{\mathbf{z}}

\NewDocumentCommand{\polydisc}{soo}{
	\IfBooleanTF{#1}					
		{\Delta_{#2}\times\Delta_{#3}}	
		{\IfNoValueTF{#2}				
			{\Delta}
			{\Delta_{#2}}
		}
}
\RenewDocumentCommand{\order}{som}{
	\IfBooleanTF{#1}
		{
		\IfNoValueTF{#2}
			{\Omega\qty(#3)}
			{\Omega_{#2}\qty(#3)}
		}	
		{
		\IfNoValueTF{#2}
			{O\qty(#3)}
			{O_{#2}\qty(#3)}
		}	
}
\NewDocumentCommand{\ball}{ooo}{
	B
	\IfValueT{#1}
		{
			\qty(\IfValueTF{#2}
					{#1, #2}
					{#1}
				\IfValueT{#3}
					{; #3})
		}
}
\NewDocumentCommand{\diam}{mo}
	{
		\operatorname{diam}\qty(#1
		\IfValueT{#2}
			{; #2})
	}
\NewDocumentCommand{\dist}{mmo}
	{
		\operatorname{dist}\qty(#1,#2
		\IfValueT{#3}
			{\,; #3})
	}
\NewDocumentCommand{\latticeComplement}{o}{
	\IfNoValueTF{#1}
		{\cComplex\setminus\zIntegers^2}
		{\cComplex\setminus{\! #1}\zIntegers^2}
}

\makeatletter
\newdimen\scalemath@axis{}
\newcommand*{\scalemath}[3]{%
  #1{%
    \mathpalette{\scalemath@aux{#2}}{#3}%
  }%
}
\newcommand*{\scalemath@aux}[3]{%
  \begingroup
    \everyvbox{}%
    \settoheight\scalemath@axis{$#2\vcenter{}$}%
    \raisebox{\scalemath@axis}{%
      \scalebox{#1}{%
        \raisebox{-\scalemath@axis}{%
          $\m@th#2#3$%
        }%
      }%
    }%
  \endgroup
}
\makeatother 

\hypersetup{pdfauthor={Oded Carmon}}

\title{Analytically generated sharply o-minimal structures}

\author{Oded Carmon}
\address{Weizmann Institute of Science, Rehovot, Israel}
\email{oded.carmon@weizmann.ac.il}
\thanks{Funded by the European Union (ERC, SharpOS, 101087910), and by the ISRAEL SCIENCE FOUNDATION (grant No. 2067/23) and by the Shimon and Golde Picker - Weizmann Annual Grant.}


\date{\today}

\begin{document}

\begin{abstract}
	We describe a class of sharply o-minimal structures, called analytically generated structures, whose definable sets and their complexity filtration are determined by the collection of definable complex cells.

	We prove a polynomially effective parameterization theorem using real complex cells for real sets definable in such structures.
	Following Binyamini--Novikov, this allows us to establish a polynomially effective version of the Yomdin--Gromov lemma on $C^r$-smooth parameterizations of definable sets, which implies Wilkie's conjecture on polylogarithmic bounds for the amount of algebraic points of bounded height and degree in the transcendental part of a definable set.

	In addition, we obtain a polynomially effective preparation theorem for definable functions, similar to the subanalytic preparation theorems of Parusinski and of Lion--Rolin.
\end{abstract}

\maketitle

\addtocontents{toc}{\protect\setcounter{tocdepth}{1}}

{\small \tableofcontents}

\section{Introduction}\label{sec:introduction}
This paper is a follow-up to~\cite{BinyaminiCarmonNovikovComplexCells}, where we extended the theory of complex cells, introduced by Binyamini--Novikov in~\cite{BinyaminiNovikov2019} in the settings of $\Ralg$ and $\Ran$, to a sharply o-minimal structure $\structure$ (admitting sharp cellular decomposition), as introduced in~\cite{BinyaminiNovikovZak2022}.
In this setting, we obtained effective versions of the main cellular preparation and parametrization theorems of~\cite{BinyaminiNovikov2019}, with bounds which are polynomial in the degrees of the relevant sets.

These complex-geometric tools are applied in~\cite{BinyaminiNovikov2019} to definable \emph{real} sets in order to obtain effective versions of the Yomdin--Gromov algebraic lemma, as well as an effective preparation theorem in the spirit of the subanalytic preparation theorems of Parusinski and Lion--Rolin (see \cite[Sections 1.1 and 4.2]{BinyaminiNovikov2019}).
However, the arguments required for these applications (most importantly, \cite[Corollaries 34 and 35]{BinyaminiNovikov2019}) may not be carried out in a general sharply o-minimal structure~$\structure$ (see \Cref{sec: difference}). 

In this paper, we consider a natural reduct of a structure $\structure$ where we may apply the results of~\cite{BinyaminiCarmonNovikovComplexCells} to real definable sets.
We call this reduct the structure \emph{analytically generated} from $\structure$ (a similar definition appeared earlier in~\cite[Section 3.3]{BinyaminiNovikov2023}).
A sharply o-minimal structure equivalent (in the sense of reduction of FD-filtrations, see~\cite[Definition 2.8]{BinyaminiCarmonNovikovComplexCells}) to its analytically generated reduct will simply be called an \emph{analytically generated structure}.

After some technical set-up in \Cref{sec: ag structures} regarding such reducts, we prove in \Cref{sec: real parameterization} the following polynomially effective parameterization result for real sets definable in an analytically generated structure, using real complex cells and real cellular maps (see \Cref{thm: real parameterization}).
For the rest of this section, let $\{\FD[\format][\degree]\}$ be an analytically generated sharply o-minimal structure with sharp cell decomposition.

{
	\renewcommand{\thetheorem}{A}
	\begin{theorem}
		Let $S_1,\dots,S_k\subset \rReal^n $ be sets in ${\structure}_{\format,\degree}$.
		Then there exists a collection of $\polyfd[\format][\degree,k,1/\sigma]$ prepared real cellular maps $\{f_j:\hExt{\cell_j}{\sigma}\to\cComplex^n\}$, for complex cells $\hExt{\cell_j}{\sigma}$.
		
		Each map $f_j$ is in ${\structure}_{\order[\format]{1},\polyfd}$, the restrictions $\{\restrict{f_j}{\rRealPos\hExt{\cell_j}{\sigma}}\}$ are compatible with each of the sets $S_i$ and we have that ${\rReal^n\subset \bigcup_j f_j(\rRealPos\cell_j)}$.
	\end{theorem}
}

This allows the arguments of~\cite{BinyaminiNovikov2019} to follow through, in many cases verbatim, and yields our two main results, a sharp Yomdin--Gromov algebraic lemma (\Cref{thm: sharp yomdin gromov}) and a preparation theorem (\Cref{thm: S-preparation}) as in~\cite{BinyaminiNovikov2019}:

{
	\renewcommand{\thetheorem}{B}
	\begin{theorem}
		Let $X\subset [0,1]^n$ be of dimension $\mu$ and in $\FD[\format][\degree]$.
		Let $r\in\nNatural$.

		Then there exists a collection of $\polyfd\cdot r^\mu$ maps $\varphi_i:(0,1)^\mu\to X$, each of them in $\FD*[\format][\degree,r]$, such that $X=\bigcup_i \varphi_i((0,1)^\mu)$.
		Furthermore, for every $\boldsymbol{\alpha}\in \nNatural^\mu$ such that $\alpha_1+\cdots+\alpha_\mu\leq r$, the partial derivative of order $\boldsymbol{\alpha}$ of every $\varphi_i$ exists and is bounded uniformly by $\alpha_1!\cdots\alpha_\mu !$.
	\end{theorem}

	\renewcommand{\thetheorem}{C}
	\begin{theorem}
		Let $f_1,\dots,f_M : \rReal^n \to\rReal$ be functions in $\FD[\format][\degree]$ and denote the coordinates of $\rReal^n$ by $\varX=(x_1,\dots,x_n)$.
	Let $\mu>0$.
	Then there is a cover of $\rReal^n$ by a collection of $\polyfd[\format][\degree,M,1/\mu]$ prepared real cellular maps $\{\phi_j:\ext{\cell_j}{1/2}\to\cComplex^n\}$, each of them in $\FD*[\format][\degree]$ and compatible with the zero-sets of the coordinate functions $x_1,\dots,x_n$, such that for each $j$ we have the following expansion of each of the functions $f_i$ in $\phi_j(\rRealPos\cell_j)$:
	\begin{equation}
		f_i(\varX)=\prod_{k=1}^n \abs{\varX_k-\theta_{j,k}(\initial{\varX}{1}{k-1})}^{\alpha_{i,j,k}}
		\cdot
		U_{i,j}(\varX),
	\end{equation}
	where $\alpha_{i,j,k}\in\qRationals$ is of size at most $\polyfd$ and the functions $\theta_{j,k},U_{i,j}:\phi_{j}(\rRealPos\cell_j)\to\rReal$ are in $\FD*[\format][\degree]$.
	The sign of $U_{i,j}\comp \phi_j$ is constant on $\rRealPos\cell_j$, and, if this sign is not $0$, we have that the diameter of  $\log \abs{U_{i,j}}\comp \phi_j (\rRealPos\cell_j)$ is less than~$\mu$.

	In addition, if $\theta_{j,k}$ is not identically $0$ over $\phi_j(\rRealPos\cell_j)$, then it is nowhere vanishing and we have 
	\begin{equation}
		\abs{\varX_k - \theta_{j,k}(\initial{\varX}{1}{k-1})}\leq \mu\abs{\varX_k}
	\end{equation}
	for all $\varX\in\phi_{j}(\rRealPos\cell_j)$.
	If, moreover, we have $\alpha_{i,j,k}\neq 0$ for some $i$, then we also have that the left-hand side of \eqref{eq: center condition} is nowhere vanishing over $\phi_j(\rRealPos\cell_j)$.
	\end{theorem}
}

A similar version of the Yomdin--Gromov lemma, in the setting of a sharply o-minimal structure with \emph{sharp derivatives}, was established by Binyamini--Novikov--Zak in~\cite[Lemma 2]{BinyaminiNovikovZak2024} and used there to prove Wilkie's conjecture on polylogarithmic bounds for the number of algebraic points on sets definable in such structures (for example, their results hold for the structure $\RrPfaff$ generated by restricted Pfaffian functions, which has sharp derivatives).
They then deduce from this Wilkie's conjecture in the original setting of $\Rexp$.

Our version of the Yomdin--Gromov lemma also yields, in essentially the same way, Wilkie's conjecture for an analytically generated structure (see \Cref{thm: Wilkie conjecture}):

{
	\renewcommand{\thetheorem}{D}
	\begin{theorem}
		Let $\{\FD[\format][\degree]\}$ be an analytically generated sharply o-minimal structure with sharp cell decomposition and let $X\in\FD[\format][\degree]$.
		Then
		\begin{equation}
			\#X^{\operatorname{trans}}(g,H)=\polyfd[\format][\degree, g, \log H].
		\end{equation}
	\end{theorem}
}

While $\RrPfaff$ and its reducts serve as the only currently known examples of sharply o-minimal reducts of $\Ran$, it is expected that certain (conjectural) sharply o-minimal structures which could be relevant to applications in number theory will be analytically generated but might not admit sharp derivatives --- for example, the structure generated by restrictions of the complex Gamma function to discs in the complex plane.
Additionally, in a subsequent paper~\cite{BinyaminiCarmonNovikovLE}, we use the framework of analytically generated structures and the corresponding preparation theorem to deduce Wilkie's conjecture for $\Rexp$ in an improved form compared to that in~\cite{BinyaminiNovikovZak2024}~--- where the relevant polynomial bounds depend only on the complexity of the set on which we count algebraic points, with no dependence on its geometry.

Throughout this paper, we use the same notation and require the same preliminaries as reviewed in~\cite[Section 2 and Appendix A]{BinyaminiCarmonNovikovComplexCells}, and refer the reader there for the definitions of a sharply o-minimal structure and sharp cellular decomposition (which we abbreviate from now on as \so-minimal and \s CD, respectively), and of complex cells and the sharp cellular parameterization theorem (which we abbreviate as~\s CPT).

\subsection{Difference from the semialgebraic and subanalytic settings}\label{sec: difference}
We now explain in more detail the problem addressed by this paper.
In~\cite[Section 4.1]{BinyaminiNovikov2019}, the authors apply the cellular preparation theorem (CPT) for complex sets (either algebraic or analytic) to real sets in the following way.

In the case of a real semialgebraic set $X$, defined in some bounded box, one holomorphically continues all polynomials appearing in the definition of $X$ to a polydisc of some fixed extension, say a $1/2$-extension.
One then applies the CPT to the corresponding collection of complex algebraic hypersurfaces to obtain a cover of the holomorphic continuation of $X$ by images of complex cells under cellular maps.
Restricting these maps to the real parts of the resulting complex cells yields a cover of the original real set $X$.

In the subanalytic case, one similarly considers the real analytic functions determining the set $X$.
By intersecting $X$ with sufficiently small boxes, one may holomorphically continue each of these functions to a small polydisc admitting a $1/2$-extension and proceed as before.
The number of such polydiscs depends on the set $X$, but we have no effective estimate for it in general, and so cannot obtain effective estimates for the resulting cover of the set $X$.
Definability of these continuations is automatic when working in $\Ran$ (using small enough polydiscs), but constructing holomorphic continuations which remain in a given reduct is more difficult (see \cite{Kaiser2016}).

Furthermore, while~\cite[Corollaries 34 and 35]{BinyaminiNovikov2019} are stated for a single (semialgebraic or subanalytic) set, we wish to obtain cellular decompositions of $\rReal^n$ compatible with a given finite collection of sets.
This may be achieved by successive application of these results to each set in the collection, but for questions of effectivity it is well known that better bounds are obtained when one treats all sets in the collection ``at once''.
In the semialgebraic case this requires little modification --- one applies the CPT once to the holomorphic continuations of all polynomials appearing in the definitions of all relevant semialgebraic sets, in some suitable polydisc.
However, in the more general analytic setting, it is not clear that a collection of (definable) analytic hypersurfaces of complex cells may be continued (definably) to some shared cell.

The constructions of \Cref{sec: real parameterization} resolve these difficulties, with \Cref{thm: real parameterization} giving a polynomially effective version of~\cite[Corollaries 34 and 35]{BinyaminiNovikov2019} for a finite collection of sets definable in an analytically generated structure.

\subsection{Acknowledgements}
I wish to thank Gal Binyamini and Dmitry Novikov for many helpful discussions.

\section{Analytically generated structures}
\label{sec: ag structures}
In order to apply the complex-geometric results of~\cite{BinyaminiCarmonNovikovComplexCells} to real sets, we consider the following class of o-minimal structures (cf. \cite[Section 3.3]{BinyaminiNovikov2023} and~\cite[Section 6]{Binyamini2024}).
We recall first that we call a set $Z\subset\cComplex^\ell$ \emph{symmetric} if it is invariant under coordinate-wise complex conjugation.
We also follow the convention in~\cite[Remark 2.22]{BinyaminiCarmonNovikovComplexCells}, where the complexity of a complex cell is measured in terms of the complexities of the graphs of the holomorphic functions determining it, rather than its complexity as a definable set.

\begin{definition}[Analytically generated substructure]
\label{def: ag structure}
Let $\{\FD[\format][\degree]\}$ be a \so-minimal structure.
Let $\an{\structure}$ be the structure generated by the collection of sets of the form $Z\cap \rReal\cell$ where $\cell$ is a real complex cell and $Z\subset\ext{\cell}{1/2}$ is a symmetric analytic hypersurface, both of them definable in $\structure$.

We equip $\an{\structure}$ with the following FD-filtration.
Let $\ext{\cell}{1/2}$ be a real complex cell and let $Z\subset\ext{\cell}{1/2}$ be a symmetric analytic hypersurface, both of them in $\FD[\format][\degree]$.
We assign the set $Z\cap \rReal\cell$ format $\format$ and degree $\degree$.
The FD-filtration on $\an{\structure}$ is then sharply generated (see~\cite[Definition 2.3]{BinyaminiCarmonNovikovComplexCells}) from these initial data 
(in particular, algebraic hypersurfaces with their usual complexity are included in the resulting filtration).
We denote the resulting structure and FD-filtration by $\{\an{\structure}_{\format,\degree}\}$ and say that it is \emph{analytically generated} by~$\{\FD[\format][\degree]\}$.
\end{definition}

\begin{remark}\label{rem: algebraic hypersurfaces}
	In the construction of \Cref{def: ag structure}, it is not crucial to explicitly add algebraic hypersurfaces to the filtration $\{\an{\structure}_{\format,\degree}\}$ since one may express any such hypersurface using sets of the form $Z\cap \rReal\cell$ as follows. 

	Consider $\{P=0\}$ for a polynomial $P\in\rReal[x_1,\dots,x_n]$.
	The relevant hypersurfaces $Z$ are obtained from the set of complex zeros of $P$ in $\cComplex^n$, where we split each $\cComplex$ factor as $\cComplex=\disc[2]\cup\puncDisc*[1]$ (cf.~\cite[Remark 2.33]{BinyaminiCarmonNovikovComplexCells}).
	Each of the resulting $2^n$ real complex cells admits a $1/2$-extension, and the complex zeros of $P$ form a symmetric analytic hypersurface of each of these extensions.
	
	The complexity of $\{P=0\}$ obtained in this way is slightly larger than required by the axioms of \so-minimality, but this technicality may be ignored by~\cite[Lemma 5.9]{BinyaminiNovikovZak2022}.
\end{remark}

\begin{remark}\label{rem: construction sequence}
	We recall that membership of a set $X\in\structure$ in one of the collections $\an{\structure}_{\format,\degree}$ corresponds to the existence of a \emph{construction sequence} expressing $X$ by iterated boolean operations and projections of algebraic hypersurfaces and of sets of the form $Z\cap \rReal\cell$ as above (by \Cref{rem: algebraic hypersurfaces}, it is enough to consider only the latter), with the complexities of these basic sets and the number of operations controlled by $\format$ and $\degree$ according to axioms (\s1) -- (\s4) in~\cite[Definition 2.3]{BinyaminiCarmonNovikovComplexCells}.
	
	We note that a set $X\in\structure$ might have lower complexity with respect to the original filtration $\{\FD[\format][\degree]\}$ than with respect to the filtration $\{\an{\structure}_{\format,\degree}\}$.
	Using the format and degree of $X$ as determined by the filtration $\{\an{\structure}_{\format,\degree}\}$ ensures that a set of low complexity admits a ``simple'' construction sequence. 
\end{remark}

In \Cref{sec: real parameterization} we show that, for a \so-minimal structure with \s CD, the construction in \Cref{def: ag structure} above is idempotent up to equivalence of FD-filtrations (see \Cref{cor: analytic generation idempotence}).
Accordingly, we make the following definition.
\begin{definition}
	\label{def: analytically generated structure}
	A \so-minimal structure $\{\FD[\format][\degree]\}$ is \emph{analytically generated} if it is equivalent to $\{\an\structure_{\format,\degree}\}$.
\end{definition}

We finish this section with several propositions which will be useful later.
Fix a \so-minimal structure $\{\FD[\format][\degree]\}$.

\begin{proposition}
	\label{prop: ag is sharp}
	The structure $\{\an\structure_{\format,\degree}\}$ is reducible to $\{\FD[\format][\degree]\}$.
	In particular, $\{\an\structure_{\format,\degree}\}$ is \so-minimal.
\end{proposition}
\begin{proof}
	This follows immediately from the construction in \Cref{def: ag structure}.
	Indeed, since the definable sets and FD-filtration of $\{\an\structure_{\format,\degree}\}$ are generated by sets in $\{\FD[\format][\degree]\}$ and their complexities up to a constant increase in format and degree due to intersecting with $\rReal^n$ for an appropriate $n$, we have that~$\an\structure_{\format,\degree} \subset \FD*[\format][\degree]$.

	Now if $S\in \an\structure_{\format,\degree}$, then $S \in \FD*[\format][\degree]$ and so $S$ has at most $\polyfd$ connected components.
	In particular, $\{\an\structure_{\format,\degree}\}$ is \so-minimal.
\end{proof}

\begin{proposition}
	\label{prop: real graphs definable}
	Let $\ext{\cell}{1/2}\subset\cComplex^\ell$ be a real complex cell and let $f:\ext{\cell}{1/2}\to\cComplex$ be a real holomorphic map, both of them in $\FD[\format][\degree]$.
	Then $\graph{\restrict{f}{\rRealPos\cell}}$, as a subset of $\rReal^\ell\times\rReal$, is in $\an{\structure}_{\order[\format]{1},\polyfd}$.
\end{proposition}
\begin{proof}
	The intersection of $\ext*{\cell\odot\cComplex}{1/2}$ with the graph of $f$ is the symmetric analytic hypersurface given by $\qty{\varZ_{\ell+1}-f(\initial{\varZ}{1}{\ell})=0}$.
	The claim then follows by intersecting with the semialgebraic set $\rRealPos^\ell\times\rReal$.
\end{proof}
By considering the components of a real cellular map, we have the following immediate corollary.
\begin{corollary}
	\label{cor: images of real cellular maps}
	Let $\ext{\cell}{1/2}\subset\cComplex^\ell$ be a real complex cell and let $f:\ext{\cell}{1/2}\to\cComplex^\ell$ be a real cellular map, both of them in $\FD[\format][\degree]$.
	Then the graph and image of $\restrict{f}{\rRealPos\cell}$ in $\rReal^\ell\times\rReal^\ell$ and $\rReal^\ell$, respectively, are in $\an{\structure}_{\order[\format]{1},\polyfd}$.
\end{corollary}

\begin{corollary}
	Let $\hExt{\cell}{\rho}$ be a real complex cell of length $\ell$ and let $Z\subset\hExt{\cell}{\rho}$ be a symmetric analytic hypersurface, both of them in $\FD[\format][\degree]$.
	Then $Z\cap\rReal\cell$ is in $\an{\structure}_{\order[\format]{1},\polyfd[\format][\degree,\rho]}$.
\end{corollary}
\begin{proof}
	Using the sharp refinement theorem \cite[Theorem 2.30]{BinyaminiCarmonNovikovComplexCells}, we obtain a real cellular cover $\qty{f_j:\ext{\cell_j}{1/2}\to\hExt{\cell}{\rho}}$ of size $\polyl\qty(\rho)$ such that each $f_j$ is in $\FD*[\format][\degree]$.
	For each of the pullbacks $\pullback{f}_jZ$, we have that $\pullback{f}_jZ\cap \rRealPos\cell_j$ is in $\an{\structure}_{\order[\format]{1},\polyfd}$.
	The result now follow from \Cref{cor: images of real cellular maps}, since we have $f_j(\rRealPos\cell_j)\subset\rReal\cell$. 
\end{proof}
By the same reasoning as for \Cref{cor: images of real cellular maps}, we now have the following.
\begin{corollary}\label{cor: rho extension complexity}
	Let $\hExt{\cell}{\rho}\subset\cComplex^\ell$ be a real complex cell and let $f:\hExt{\cell}{\rho}\to\cComplex^\ell$ be a real cellular map, both of them in $\FD[\format][\degree]$.
	Then the graph and image of $\restrict{f}{\rRealPos\cell}$ in $\rReal^\ell\times\rReal^\ell$ and $\rReal^\ell$, respectively, are in $\an{\structure}_{\order[\format]{1},\polyfd[\format][\degree,\rho]}$.
\end{corollary}

Finally, as in~\cite[Section 8]{Binyamini2024}, the next proposition shows that the real complex cells in $\structure$ themselves (rather than just their real parts) are definable in $\an{\structure}$, under the identification $\cComplex^\ell\cong\rReal^{2\ell}$.

\begin{proposition}
	\label{prop: cells are definable}
Let $\hExt{\cell}{\rho}\subset\cComplex^\ell$ be a real complex cell and $f:\hExt{\cell}{\rho}\to\cComplex$ be a real holomorphic function, both of them in $\FD[\format][\degree]$.
Then the graph of $\restrict{f}{\cell}$, as a subset of $\cComplex^\ell\times\cComplex$, is in $\an{\structure}_{\order[\format]{1},\polyfd[\format][\degree,\rho]}$.
\end{proposition}
\begin{proof}
	The proof is the same as that of~\cite[Proposition 48]{Binyamini2024}.
	The technical verifications needed to carry out the arguments in loc.\ cit.\ (for example, the use of monomial cells as in~\cite[Section 6.2]{BinyaminiNovikov2019}) follow from the previous propositions of this section. 
\end{proof}

\section{Parametrization of real sets}
\label{sec: real parameterization}
The main result of this section, \Cref{thm: real parameterization}, is the required analog of~\cite[Corollary 34]{BinyaminiNovikov2019} needed to deduce our main theorems, \Cref{thm: sharp yomdin gromov,thm: Wilkie conjecture,thm: S-preparation} (see \Cref{def: cylindrical collection} for the notion of a cylindrical collection of cellular maps).
\begin{theorem}
	\label{thm: real parameterization}
	Let $\{\FD[\format][\degree]\}$ be a \so-minimal structure with \s CD.
	Let $S_1,\dots,S_k\subset \rReal^n $ be sets in $\an{\structure}_{\format,\degree}$.
	Then there exists a cylindrical collection of prepared real cellular maps $\{f_j:\hExt{\cell_j}{\sigma}\to\cComplex^n\}$ of size $\polyfd[\format][\degree,k,1/\sigma]$, where each $f_j$ is in ${\structure}_{\order[\format]{1},\polyfd}$, the restrictions $\{\restrict{f_j}{\rRealPos\hExt{\cell_j}{\sigma}}\}$ are compatible with each $S_i$ and we have that $\rReal^n\subset \bigcup_j f_j(\rRealPos\cell_j)$.
\end{theorem}

Before giving the proof of \Cref{thm: real parameterization}, we first consider the following sequence of corollaries.

\begin{corollary}
	\label{cor: ag has CD}
	Let $\{\FD[\format][\degree]\}$ be a \so-minimal structure with \s CD.
	Then the structure $\{\an{\structure}_{\format,\degree}\}$ is \so-minimal and has \s CD.
\end{corollary}
\begin{remark}
	We note that we have already established that $\{\an{\structure}_{\format,\degree}\}$ is \so-minimal in \Cref{prop: ag is sharp}, however this will also follow from the proof below.
\end{remark}
\begin{proof}[Proof of \Cref{cor: ag has CD}]
	Let $S\in\an{\structure}_{\format,\degree}$.
	Apply \Cref{thm: real parameterization} and let $\{f_j:\ext{\cell_j}{1/2}\to\cComplex^\ell\}$ be the resulting collection of prepared real cellular maps.
	By \Cref{cor: images of real cellular maps}, the restrictions $\restrict{f_j}{\rRealPos\cell_j}$ are in $\an{\structure}_{\order[\format]{1},\polyfd[\format][\degree]}$.
	The images of these maps are connected, which implies that $S$ has $\polyfd$ connected components, each of them in~$\an{\structure}_{\order[\format]{1},\polyfd[\format][\degree]}$.
	The result now follows by~\cite[Proposition 1.23]{BinyaminiNovikovZak2022}.
\end{proof}

\begin{corollary}
	\label{cor: analytic generation idempotence}
	Let $\{{\widetilde{\structure}}_{\format,\degree}\}$ be a \so-minimal structure with \s CD and let $\{\FD[\format][\degree]\}=\{\an{\widetilde{\structure}}_{\format,\degree}\}$.
	Then $\{\FD[\format][\degree]\}$ and $\{\an\structure_{\format,\degree}\}$ are equivalent.
	In particular, the structures $\structure$ and $\an{\structure}$ are equal as collections of sets.
\end{corollary}

\begin{proof}
	By \Cref{prop: ag is sharp}, we have that $\{\FD[\format][\degree]\}$ is reducible to $\{{\widetilde{\structure}}_{\format,\degree}\}$, that $\{\an\structure_{\format,\degree}\}$ is reducible to $\{\FD[\format][\degree]\}$, and that $\{\FD[\format][\degree]\}$ is \so-minimal.

	It remains to show that $\{\FD[\format][\degree]\}$ is reducible to $\{\an\structure_{\format,\degree}\}$.
	Let $S\in \FD[\format][\degree]$.
	We apply \Cref{thm: real parameterization} with respect to $\{{\widetilde{\structure}}_{\format,\degree}\}$ and let $\{f_j:\ext{\cell_j}{1/2}\to\cComplex^\ell\}$ be the resulting collection of prepared real cellular maps.
	By \Cref{prop: cells are definable}, we have $f_j\in\FD*[\format][\degree]$.
	Thus, by \Cref{cor: images of real cellular maps}, the restrictions $\restrict{f_j}{\rRealPos\cell_j}$ are in $\an{\structure}_{\order[\format]{1},\polyfd[\format][\degree]}$.
	Since $S$ is the union of images of $\polyfd$ of these restrictions, we have $S\in \an{\structure}_{\order[\format]{1},\polyfd[\format][\degree]}$, which finishes the proof.
\end{proof}

From the last two corollaries together we get the following (see \Cref{def: analytically generated structure}).
\begin{corollary}
	\label{cor: ag so cd inheritence}
	Let $\{{{\structure}}_{\format,\degree}\}$ be a \so-minimal structure with \s CD.
	Then $\{\an{{\structure}}_{\format,\degree}\}$ is an analytically generated \so-minimal structure with \s CD.
\end{corollary}

\subsection{Proof of the parametrization theorem}
We now turn to the proof of \Cref{thm: real parameterization}.
The proof in our case is complicated by the fact that, unlike~\cite[Corollary 34]{BinyaminiNovikov2019}, the sets to be parameterized are not necessarily expressed in terms of functions which continue holomorphically to a shared polydisc.
Thus, we must construct cells on which the functions (locally) defining the sets to be parametrized may be simultaneously considered.

A similar problem is addressed in~\cite[Section 6.2]{Binyamini2024}.
The construction given there, while effective, does not guarantee polynomial dependence on the degrees of the relevant sets.
We give a construction which yields the desired size and complexity bounds.

For the rest of this section, we fix a \so-minimal structure $\{\FD[\format][\degree]\}$ with \s CD.
We introduce the following notions of a simultaneous cellular cover and of a cylindrical collection of real cellular maps.

\begin{definition}[Simultaneous cellular cover]
	Let $\rho,\sigma>0$ and let $\{\hExt{\cell_i}{\rho}\}$ be a finite collection of real complex cells of length $\ell$.
	A finite collection of real cellular maps $\{f_j:\hExt{\widehat\cell_j}{\sigma}\to\cComplex^\ell\}$ is a \emph{simultaneous cellular cover} of $\{\hExt{\cell_i}{\rho}\}$ if the following two conditions hold.
	\begin{itemize}
		\item For each $i$, there is a subcollection of $\{f_j\}$ which is a real cellular cover of $\hExt{\cell_i}{\rho}$.
		\item For all $i,j$, if the intersection $f_j(\hExt{\widehat\cell_j}{\sigma})\cap\cell_i$ is not empty, then $f_j(\hExt{\widehat\cell_j}{\sigma})\subset \hExt{\cell_i}{\rho}$.
	\end{itemize}
\end{definition}
It is easy to check that, in the notation above, replacing one of the maps $f_j$ by its composition with a real cellular cover of $\hExt{\widehat\cell_j}{\sigma}$ yields another simultaneous cellular cover of $\{\hExt{\cell_i}{\rho}\}$.

\begin{definition}[Cylindrical collection of cellular maps]
	\label{def: cylindrical collection}
	Let $\{f_j:\cell_j\odot\cell[F]_j\to\cComplex^\ell\}$ be a collection of real cellular maps, for $j$ in some index set $J$.
	The collection is \emph{cylindrical} if either $\ell=0$ or the following two conditions hold:
	\begin{itemize}
		\item The collection $\{\initial{(f_j)}{1}{\ell-1}:\cell_j\to\cComplex^{\ell-1}\}$ is cylindrical;
		\item Let $j\in J$ and let $J'\subset J$ be the collection of those $j'\in J$ such that $\cell_{j'}=\cell_{j}$ and $\initial{(f_{j'})}{1}{\ell-1}=\initial{(f_j)}{1}{\ell-1}$.
		We have
		\begin{equation}
			\bigcup_{j'\in J'}f_{j'}(\rRealPos(\cell_{j'}\odot\cell[F]_{j'}))=\initial{(f_j)}{1}{\ell-1}(\rRealPos\cell_j)\times \rReal.
		\end{equation}
	\end{itemize}
\end{definition}

\begin{remark}
	We note that the conditions in \Cref{def: cylindrical collection} hold automatically for a collection of real cellular maps $\{f_j\}$ such that the restrictions $\{\initial{(f_j)}{1}{k}\}$ are disjoint real cellular covers for all $k$.
	However, composing one of the maps $f_j$ with a real cellular cover does not, in general, preserve cylindricity.
	Therefore, from now on whenever we apply the refinement theorem, the \s CPT or the \s CPrT~\cite[Theorems 2.30, 2.37, and 2.39]{BinyaminiCarmonNovikovComplexCells} to a cell in a cylindrical collection, we will implicitly mean that we do it in the following ``cylindrical'' manner.

	In the proof of the refinement theorem (see~\cite[Section 6.1]{BinyaminiNovikov2019}), whenever we would inductively apply the refinement theorem to the base $\cell$ of a cell $\cell\odot\cell[F]_j$, we use the same cellular cover of $\cell$ for all cells $\cell\odot\cell[F]_j$ sharing this base (if different $j$ call for different extension parameters in this inductive application, we use the smallest parameter among those required).
	We apply the same cylindrical procedure in this inductive application of the refinement theorem to the base~$\cell$.
		
	We make the same kind of modification whenever, in the proof of the \s CPT, we would inductively apply the \s CPT to the base $\cell$ of a cell $\cell\odot\cell[F]_j$.
	In this case, we consider the collection consisting of all relevant analytic hypersurfaces of $\cell$ corresponding to all cells $\cell\odot\cell[F]_j$ sharing this base.
	The proof of the \s CPrT may be treated in the same way.

	It is straightforward to verify that these modifications still yield cellular covers with the desired size and complexity bounds.
	Furthermore, composing the cylindrical collection $\{f_j:\hExt{\cell_j}{\rho}\to\cComplex^\ell\}_{j\in J}$ with cellular covers of the cells $\{\hExt{\cell_j}{\rho}\}$ constructed in this way yields a cylindrical collection of cellular maps.
\end{remark}

The motivation for the notion of a cylindrical collection of cellular maps comes from the following lemma.
\begin{lemma}
	\label{lem: cylindrical compatibility}
	Let $\{f_j:\cell_j\to\cComplex^{\ell+1}\}$ be a cylindrical collection of real cellular maps such that the maps $\{\restrict{f_j}{\rRealPos\cell_j}\}$ are compatible with some non-empty set $S\subset \rReal^{\ell+1}$.
	Then, for every $1\leq k\leq\ell$, every map in the collection $\{\initial{(f_j)}{1}{k}:\initial{(\rRealPos\cell_j)}{1}{k}\to\rReal^k\}$ is compatible with the projection $\initial{\pi}{1}{k}(S)\subset\rReal^k$ of $S$ onto its first $k$ coordinates.
\end{lemma}
\begin{proof}
	Since $\{\initial{(f_j)}{1}{\ell}\}$ is a cylindrical collection whenever $\{f_j\}$ is, it is enough to treat the case $k=\ell$.
	Let $j$ be such that $\initial{(f_{j})}{1}{\ell}(\initial{(\rRealPos\cell_{j})}{1}{\ell})\cap \initial{\pi}{1}{\ell}(S)\neq\emptyset$, i.e.\ there exists $\initial{\varX}{1}{\ell}\in \initial{(\rRealPos\cell_j)}{1}{\ell}$ and $y\in\rReal$ satisfying $(\initial{(f_j)}{1}{\ell}(\initial{\varX}{1}{\ell}),y)\in S$.
	Then, since the collection $\{f_j\}$ is cylindrical, there exist $j'$ and $\varX\in \rRealPos\cell_{j'}$ such that $\initial{(\cell_j)}{1}{\ell}=\initial{(\cell_{j'})}{1}{\ell}$ and $f_{j'}(\varX)=(\initial{(f_j)}{1}{\ell}(\initial{\varX}{1}{\ell}),y)\in S$.
	Since the maps $\{\restrict{f_j}{\rRealPos\cell_j}\}$ are compatible with $S$, we have that $f_{j'}(\rRealPos\cell_{j'})\subset S$.
	In particular, we have
	\begin{align}
		\begin{split}
		\initial{(f_{j})}{1}{\ell}(\initial{(\rRealPos\cell_{j})}{1}{\ell}) &= \initial{(f_{j'})}{1}{\ell}(\initial{(\rRealPos\cell_{j'})}{1}{\ell})\\
		&\subset \projbase(S)
		\end{split}
	\end{align}
	as required.
\end{proof}

\begin{corollary}
	\label{cor: cylindrical cover and generated sets}
	Let $\{f_j:\cell_j\to\cComplex^{\ell+1}\}$ be a cylindrical collection of cellular maps, such that the maps $\{\restrict{f_j}{\rRealPos\cell_j}\}$ are compatible with sets $S_1,\dots, S_k\subset \rReal^{\ell+1}$.
	Then the restrictions of the maps $\{f_j\}$ to the real parts of the bases of the cells $\cell_j$ are compatible with any set obtained from the sets $\{S_i\}$ by iterated projections, negations and boolean combinations.
\end{corollary}

Finally, we require the following two lemmas.
\Cref{lem: absorb product with R}, in light of \Cref{cor: cylindrical cover and generated sets} above, allows us to pass from general definable sets to analytic hypersurfaces of complex cells.

\Cref{lem: simultaneous cover} resolves the difference between our setting and that of~\cite[Corollaries 34 and 35]{BinyaminiNovikov2019} by constructing simultaneous cellular covers of the appropriate size and complexity for a given collection of real complex cells.

\begin{lemma}
	\label{lem: absorb product with R}
	Let ${S_1,\dots,S_k\in \an{\structure}_{\format,\degree}}$.
	Then there exists a finite collection of symmetric analytic hypersurfaces $Z_{i,j}$, contained in real complex cells $\ext{\cell_{i,j}}{1/2}$, such that $S_i$ is obtained by iterated projections, negations and boolean combinations of the sets $Z_{i,j}\cap\rReal\cell_{i,j}$.

	We may take the number of these hypersurfaces to be $\polyfd[\format][\degree,k]$, the length of the complex cell $\cell_{i,j}$ to be $\ell_{i,j}=\order[\format]{1}$, and each $Z_{i,j}$ and $\cell_{i,j}$ to be in~${\structure}_{\order[\format]{1},\polyfd}$.
\end{lemma}
\begin{proof}
	It suffices to treat the case $k=1$.
	Let $S\in\an{\structure}_{\format,\degree}$.
	As in \Cref{rem: algebraic hypersurfaces,rem: construction sequence}, we have that $S$ is obtained by a sequence of applications of the axioms (\s 1)\nobreakdash --(\s 4) of \so-minimality, as described in~\cite[Definition 2.3]{BinyaminiCarmonNovikovComplexCells}, to generating sets of the form $Z\cap\rReal \cell$, for an analytic hypersurface $Z$ and a real complex cell $\cell$ as in \Cref{def: ag structure}.
	We proceed by induction on the number of these applications, which is determined by $\format$ and $\degree$, where the base case corresponds to one of the sets $Z\cap\rReal \cell$ (cf.\ the notion of \emph{structure trees} in \cite[Section 5.2]{BinyaminiNovikovZak2022}).

	Denote by $r$ the number of steps in such a sequence corresponding to forming a cartesian product with $\rReal$ on the right.
	Since $r\leq \format$, it is enough to prove the lemma with the bounds $\polyfd$ and $\order[\format]{1}$ replaced by $\polyfd[\format,r][\degree]$ and~$\order[\format,r]{1}$, respectively.

	If $S=Z\cap \rReal\cell$ for a symmetric analytic hypersurface $Z$ of a real complex cell $\ext{\cell}{1/2}$, the result is clear.
		
	For the inductive step, we note that the result follows immediately for complements, projections and (finite) unions and intersections of sets for which the conclusion of the lemma is known to hold.
	It is thus enough to consider the following case.
	Let $S\subset\rReal^n$ be in $\an{\structure}_{\format,\degree}$ and satisfying the conclusion of the lemma.
	Assume also that its corresponding construction sequence  has less than $r$ cartesian products with $\rReal$ on the right.
	We wish to establish the conclusion of the lemma for the sets $\rReal\times S$ and $S\times\rReal$.

	If $S=Z\cap \rReal\cell$ in the notation as above, then the result follows by noting
	\begin{align}
		\begin{split}
			\rReal\times (Z\cap \rReal\cell)&=(\cComplex\times Z)\cap\rReal(\cComplex\odot\cell),\\
			(Z\cap \rReal\cell)\times\rReal&=(Z\times \cComplex)\cap\rReal(\cell\odot \cComplex).
		\end{split}
	\end{align}

	It is easy to check that taking a product with $\rReal$ (on either side) commutes with complements, intersections and unions, and that taking a product with $\rReal$ on the left commutes with projection operators.
	Hence it remains to consider the following case.
	Let $S\subset\rReal^n$ be as above and consider the set $\pi(S)\times\rReal$, where $\pi$ is the natural projection operator omitting the last coordinate.
	This set is equal to
	\begin{equation}
		\{(x_1,\dots,x_{n})\st\exists x_{n+1}.(x_1,\dots,x_{n-1},x_{n+1}\in S)\}.
	\end{equation}
	This is the same as the set obtained by applying the projection operator $n+1$ times to
	\begin{equation}
		(\rReal^{n+1}\times S)\cap\{x_1=x_{n+2}\}\cap\cdots\cap\{x_{n-1}=x_{2n}\}\cap\{x_{n+1}=x_{2n+1}\},
	\end{equation}
	which reduces this case to the previously considered cases.
\end{proof}

\begin{lemma}
	\label{lem: simultaneous cover}
	Let $\sigma>0$ and let $\{\ext*{\cell_i\odot\cell[F]_i}{1/2}\}_{i=1}^k$ be a finite collection of real complex cells of length $\ell+1$, each of them in $\FD[\format][\degree]$.
	Then there exists a cylindrical simultaneous cellular cover $\{f_j:\hExt*{\widehat{\cell}_j\odot\widehat{\cell[F]}_j}{\sigma}\to\cComplex^{\ell+1}\}$ of $\{\ext*{\cell_i\odot\cell[F]_i}{1/2}\}$, which is of size $\polyfd[\format*][\degree,k,1/\sigma]$ and such that each $f_j$ is in $\FD*[\format*][\degree]$.
\end{lemma}
\begin{proof}
	We may assume without loss of generality that the collection $\{\cell_i\odot\cell[F]_i\}$ contains the $2^{\ell+1}$ real complex cells obtained by writing each coordinate of $\cComplex^{\ell+1}$ as $\disc[2]\cup\puncDisc*[1]$ as in~\cite[Remark 2.33]{BinyaminiCarmonNovikovComplexCells}.
	By induction on $\ell$, we may assume that we have already constructed a cylindrical simultaneous cellular cover ${\{f_j:\ext{\widehat{\cell}_j}{1/2}\to\cComplex^\ell\}}$ of the collection $\{\ext{\cell_i}{1/2}\}$ with the stated size and complexity bounds.
	Let $\Sigma$ be the collection of pairs $(i,j)$ such that $f_j(\ext{\widehat\cell_j}{1/2})\subset \ext{\cell_i}{1/2}$.
	For $(i,j)\in\Sigma$, the radii defining $\cell[F]_i$ pull back to real holomorphic functions defined on $\ext{\widehat\cell_j}{1/2}$.
	By renaming, we will assume from now on that these radii and fibers are defined over $\ext{\widehat\cell_j}{1/2}$.

	As in~\cite[Section 3.1]{BinyaminiCarmonNovikovComplexCells}, we may apply the \s CPT to $\ext{\widehat\cell_j}{1/2}$ and an appropriate discriminant set in order to reduce to the case where the radii defining the fibers $\cell[F]_i$ are pairwise distinct over $\ext{\widehat\cell_j}{1/2}$, for all $(i,j)\in\Sigma$.
	By this application of the \s CPT, we may also assume that the cells $\widehat\cell_j$ admit $\hyperbolicParameter{\sigma'}$-extensions for some value of $\sigma'$ to be specified later, as long as $1/\sigma'=\polyfd[\format,\ell][\degree,k]$.
	We will assume for now only that $\hyperbolicParameter{\sigma'}<1/2$.
	In particular, this upper bound allows us to suppress the dependence on $\sigma'$ from our estimates, by \Cref{cor: rho extension complexity}.

	We apply the clustering constructions of~\cite[Section 6.3]{BinyaminiNovikov2019} to the radii of the fibers $\cell[F]_i$, viewed as functions defined on $\hExt{\widehat\cell_j}{\sigma'}$ (we note that, since these radii are all already univalued, there is no need to pass to a $\nu$-cover in the sense of~\cite[Section 2.6]{BinyaminiNovikov2019}).
	We cluster these radii around $0$ over the base $\ext{\widehat\cell_j}{1/2}$ using a gap parameter $\gamma=1-\frac{1}{\order[\ell]{k}}$ such that $\gamma^{\order[\ell]{k}}>1/2$.
	By~\cite[Proposition 56]{BinyaminiNovikov2019}, this is possible with an appropriate choice of $1/\sigma'=\polyfd[\ell][k]$.
	In the notation of~\cite{BinyaminiNovikov2019}, we obtain the clustering fibers $\{\cell[F]_{0,q}\}$ and $\{\cell[F]_{0,q+}\}$ which are well defined over the base $\ext{\widehat\cell_j}{1/2}$ and together cover $\cComplex$.
	Let $\widehat{\cell[F]}$ denote one of the fibers constructed in this way and assume its $\gamma$-extension intersects the fiber $\cell[F]_i$.
	Our choice of $\gamma$ and~\cite[Proposition 56]{BinyaminiNovikov2019} imply in this case that $\ext{\widehat{\cell[F]}}{\gamma}$ is contained in $\ext{\cell[F]_i}{1/2}$, uniformly over the base $\ext{\widehat\cell_j}{1/2}$.

	The collection of maps $f_j\odot \idmap:\ext{\widehat{\cell}_j}{1/2}\odot\ext{\widehat{\cell[F]}}{\gamma}\to \cComplex^{\ell+1}$ is thus a cylindrical simultaneous cover of $\{\ext*{\cell_i\odot\cell[F]_i}{1/2}\}$.
	Finally, by using the sharp refinement theorem, we refine the cells $\ext{\widehat{\cell}_j}{1/2}\odot\ext{\widehat{\cell[F]}}{\gamma}$ to cells admitting $\hyperbolicParameter{\sigma}$-extensions.
\end{proof}

\begin{proof}[Proof of \Cref{thm: real parameterization}]
	By the \s CPrT~\cite[Theorem 2.39]{BinyaminiCarmonNovikovComplexCells}, it is enough to consider the case where $\hyperbolicParameter{\sigma}=1/2$.
	By \Cref{lem: absorb product with R}, there exist $\polyfd[\format][\degree,k]$ symmetric analytic hypersurfaces $Z_{i,j}$ lying in real complex cells $\ext{\cell_{i,j}}{1/2}\subset\cComplex^{\ell_{i,j}}$, where $\ell_{i,j}=\order[\format]{1}$, such that each of the sets $S_i$ in the statement of the theorem is obtained by iterated projections, complements, and boolean combinations of the sets $Z_{i,j}\cap\rReal\cell_{i,j}$ and such that $Z_{i,j}$ and $\cell_{i,j}$ are in $\FD*[\format][\degree]$.

	We may assume all $\ell_{i,j}$ are equal to $\ell=\order[\format]{1}$ by taking cartesian products of the relevant cells and analytic hypersurfaces with copies of $\cComplex$.
	Similarly, we may assume without loss of generality that the hypersurfaces $Z_{i,j}$ include also all coordinate hyperplanes and the graphs of all radii functions defining the complex cells $\cell_{i,j}$, as well as their additive inverses (cf.\ the notion of boundary equations in~\cite[Section 6.1]{Binyamini2024}).
	Let $Z_\alpha\subset \ext{\cell_\alpha}{1/2}$ be the resulting collection of hypersurfaces and real complex cells.
	By \Cref{cor: cylindrical cover and generated sets}, it is enough to consider the case where the collection $\{S_i\}$ is replaced by the collection $\{Z_{\alpha}\cap\rReal\cell_{\alpha}\}$.

	Applying \Cref{lem: simultaneous cover}, we obtain a cylindrical simultaneous cover $\{f_{\beta}:\ext{\widehat{\cell}_\beta}{1/2}\to\cComplex^\ell\}$ of $\{\ext{\cell_\alpha}{1/2}\}$ which is of size $\polyfd[\format][\degree,k]$ and such that each $f_{\beta}$ is in $\FD*[\format][\degree]$.
	We pull back to each $\ext{\widehat{\cell}_\beta}{1/2}$ those analytic hypersurfaces $Z_\alpha$ corresponding to cells $\ext{\cell_\alpha}{1/2}$ such that $f_\beta(\ext{\widehat{\cell}_\beta}{1/2})\subset \ext{\cell_\alpha}{1/2}$.

	We may now apply the real \s CPT in each cell $\ext{\widehat{\cell}_\beta}{1/2}$ with respect to the corresponding collection of $\polyfd[\format][\degree,k]$ analytic hypersurfaces to obtain the desired cellular cover.
	Indeed, if $g:\hExt{\cell}{\sigma}\to\cComplex^\ell$ is one of the resulting real cellular maps and $g(\rRealPos\hExt{\cell}{\sigma})\cap (Z_\alpha\cap\rReal\cell_\alpha)\neq \emptyset$, then the compatibility of the maps constructed in the \s CPT with the relevant pullback of $Z_\alpha$ implies that $g(\hExt{\cell}{\sigma})\subset Z_\alpha$ and the compatibility with the pullbacks of the radii functions determining $\cell_\alpha$ implies that $g(\rRealPos\hExt{\cell}{\sigma})\subset \rReal\cell_\alpha$.
\end{proof}

\section{Smooth parameterizations and point counting}
Using \Cref{thm: real parameterization} in place of~\cite[Corollary 34]{BinyaminiNovikov2019}, the arguments of~\cite[Sections 9 and 10]{BinyaminiNovikov2019} extend directly to the \so-minimal setting upon replacing all bounds of the forms $\polyl(\beta)$ and $\order[\ell]{1}$ by $\polyfd[\format,\ell][\degree]$ and $\order[\format,\ell]{1}$, respectively.

In particular, we obtain the following version of~\cite[Theorem 1]{BinyaminiNovikov2019}, establishing a \so-minimal version of the Yomdin--Gromov lemma on $C^r$-smooth parameterizations (see~\cite[3.3]{Gromov1987}) for analytically generated structures, with polynomial dependence on $r$ and on the degrees of the parameterized sets.

\begin{theorem}
	\label{thm: sharp yomdin gromov}
	Let $\{\FD[\format][\degree]\}$ be an analytically generated \so-minimal structure with \s CD.
	Let $X\subset [0,1]^n$ be of dimension $\mu$ and in $\FD[\format][\degree]$.
	Let $r\in\nNatural$.
	Then there exists a collection of $\polyfd\cdot r^\mu$ maps $\varphi_i:(0,1)^\mu\to X$, each of them in $\FD*[\format][\degree,r]$, such that $X=\bigcup_i \varphi_i((0,1)^\mu)$.
	Furthermore, for every $\boldsymbol{\alpha}\in \nNatural^\mu$ such that $\alpha_1+\cdots+\alpha_\mu\leq r$, the partial derivative of order $\boldsymbol{\alpha}$ of every $\varphi_i$ exists and is bounded uniformly by $\alpha_1!\cdots\alpha_\mu !$.
\end{theorem}

This implies Wilkie's conjecture on polylogarithmic point counting for $\structure$, as we now explain.
For $X\subset\rReal^n$, we denote by $X(g,H)$ the set of algebraic points in $X$ with degree at most $g$ and multiplicative Weil height at most $H$.
We write $X^{\operatorname{alg}}$ for the union of all connected positive-dimensional semialgebraic subsets of $X$, and set $X^{\operatorname{trans}}=X\setminus X^{\operatorname{alg}}$.
For finite $X$, we write $\# X$ for the number of points in $X$.

\begin{theorem}
	\label{thm: Wilkie conjecture}
	Let $\{\FD[\format][\degree]\}$ be an analytically generated \so-minimal structure with \s CD and let $X\in\FD[\format][\degree]$.
	Then
	\begin{equation}
		\#X^{\operatorname{trans}}(g,H)=\polyfd[\format][\degree, g, \log H].
	\end{equation}
\end{theorem}
\begin{proof}
	The proof is the same as that of~\cite[Theorem 1]{BinyaminiNovikovZak2024}, using \Cref{thm: sharp yomdin gromov} above in place of Lemma 2 of loc.\ cit.
\end{proof}

\begin{remark}
	We note that one may also obtain \Cref{thm: Wilkie conjecture} by a more direct ``complex cellular'' approach, similar to that of~\cite[Appendix B.1]{BinyaminiNovikov2019}, without using $C^r$-smooth parameterizations.
	The argument presented there, using a complex cellular version of the Bombieri--Pila determinant method, is only strong enough to give good bounds in the case $g=1$.
	The general case follows by a similar approach, replacing the use of interpolation determinants by the construction of an auxiliary polynomial whose size is bounded from above in terms of its degree and the height of its coefficients.
	This upper bound is then compared against a lower bound coming from Liouville's inequality to obtain the main interpolation result needed for the inductive step of the proof.

	We follow this strategy in~\cite{BinyaminiCarmonNovikovLE} to establish Wilkie's conjecture for the structure $\structure(\exp)$, obtained from an analytically generated \so-minimal structure $\structure$ (containing $\restrict{\exp}{[0,1]}$) by adjoining the graph of the unrestricted real exponential function.
	In this setting, we do not have an analog of \Cref{thm: sharp yomdin gromov} and additional features of the complex cellular approach turn out to be crucial.	
\end{remark}

\section{Preparation theorem}
\label{sec: preparation theorem}
Let $\{\FD[\format][\degree]\}$ be an analytically generated \so-minimal structure with \s CD.
Similarly to~\cite[Section 4.2]{BinyaminiNovikov2019}, we may use \Cref{thm: real parameterization} in place of~\cite[Corollaries 34 and 35]{BinyaminiNovikov2019} to obtain sharp versions of the subanalytic preparation theorems of Parusinski~\cite{Parusinski1994} and Lion--Rolin~\cite{LionRolin1997} (see also~\cite[Theorem 2.3]{vdDriesSpeisseger2002}).
We explain this in this section.

The following is an analog of the monomialization lemma of~\cite[Lemma 17]{BinyaminiNovikov2019}.
\begin{lemma}[Monomialization lemma]
	\label{lem: monomialization}
	Let $\rho>0$, let $\hExt{\cell}{\rho}$ be a complex cell of length $\ell$ and let $f:\hExt{\cell}{\rho}\to\cComplex\setminus\{0\}$ be a holomorphic map in $\FD[\format][\degree]$.
	Then $f=\varZ^{\boldsymbol{\alpha}}\cdot U(\varZ)$, where $\boldsymbol{\alpha}\in\zIntegers^\ell$ is such that $\abs{\boldsymbol{\alpha}}<\polyfd$ and where $\log U:\cell\to\cComplex$ is univalued and satisfies
	\begin{align}
		\diam{\log U(\cell)}[\cComplex]&<\polyfd\cdot\rho,\\
		\diam{\Im\log U(\cell)}[\rReal]&<\polyfd.
	\end{align}
	Furthermore, the $i$-th coordinate of $\boldsymbol{\alpha}$ is $0$ whenever the $i$-th coordinate of $\cell$ is of type $\point$ or $\disc$.
\end{lemma}
\begin{proof}
	The proof in~\cite[Section 5.5]{BinyaminiNovikov2019} extends directly to the \so-minimal setting.
\end{proof}

We also need the following observation on the prepared maps constructed in the proof of the \s CPT.
\begin{remark}
	\label{rem: no bad annuli}
	In the clustering procedure used in the proofs of the CPT and \s CPT (see~\cite[Section 8.1]{BinyaminiNovikov2019} and~\cite[Section 3.1]{BinyaminiCarmonNovikovComplexCells}), we consider an analytic hypersurface $Z$ in a complex cell and its proper unramified projection $\projbase(Z)$ to the base of this cell.
	We cover the complement to the sections $\{y_i\}$ of this projection (over each point in the base) by \emph{Voronoi cells}, which are translates of complex cells of length $1$, admitting a suitable $\delta$-extension and constructed as follows.
	
	Let $y_{i_0}\in\{y_i\}$ be one of the sections as above.
	Centered around $y_{i_0}$, we have a collection of annuli, punctured discs and disc complements $\{\cell[F]_{i_0,q}\}$.
	This collection partitions a subset of $\{y_i\}$ into \emph{clusters} around $y_{i_0}$ --- we say that two sections are in the same cluster if they lie in the empty region between, say, two consecutive annuli in this collection.
	In addition, we have a collection of discs which cover some portion of these empty regions and whose extensions do not intersect any of the sections $\{y_i\}$ (see \Cref{fig: clustering}).

	We first cluster all non-zero sections around the section $y_0=0$.
	Fibers of type $\puncDisc*$ may only be constructed in this stage.
	All subsequent clustering around any other section $y_i\neq 0$ only involves a small subset of sections which are close to $y_i$ and which are clustered using punctured discs and annuli.
	In particular, any resulting fiber of type $\annulus$ (along with its $\delta$-extension) which is not centered at the origin does not wind around the origin.
	That is, if we replace this annulus by the disc determined by its outer radius, the $\delta$-extension of this disc does not intersect the origin.
\end{remark}

\begin{figure}
	\includegraphics[width=0.66\linewidth]{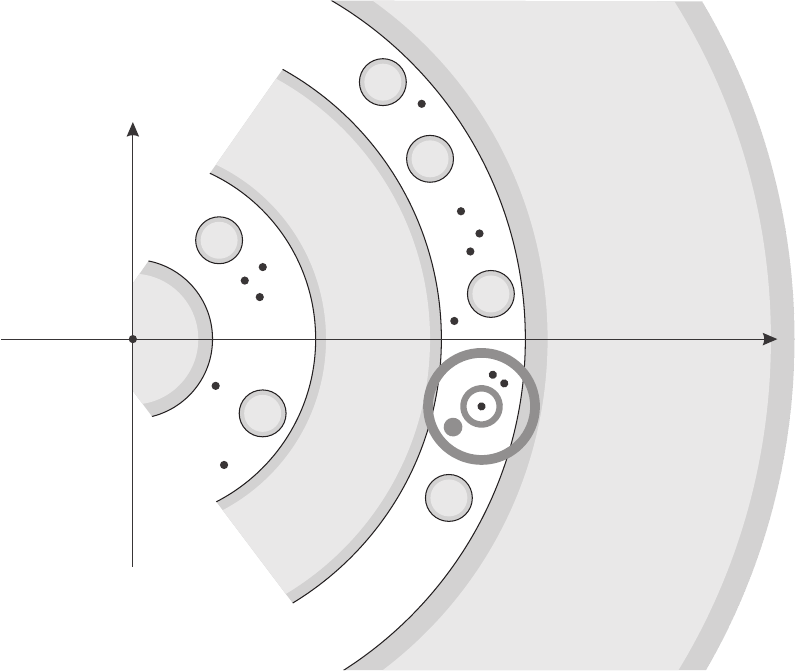}
	\caption{\emph{Clustering sections of the projection $\pi:Z \to \initial{\cell}{1}{\ell}$.}
	Above each point $z\in\initial{\cell}{1}{\ell}$, the complement of the sections $\pi^{-1}(z)$ is covered by discs and annuli whose extensions do not intersect $\pi^{-1}(z)$.
	Annuli centered around one of the sections (e.g.\ the origin) group the remaining sections into clusters.
	The region between two such annuli is covered by discs and by similar configurations of discs and annuli centered around other sections in the cluster.}
	\label{fig: clustering}
\end{figure}

We are now ready to state the main theorem of this section.
For $\varX=(x_1,\dots,x_n)$, we write $\initial{\varX}{1}{k}$ for $(x_1,\dots,x_k)$.
\begin{theorem}
	\label{thm: S-preparation}
	Let $f_1,\dots,f_M : \rReal^n \to\rReal$ be functions in $\FD[\format][\degree]$ and denote the coordinates of $\rReal^n$ by $\varX=(x_1,\dots,x_n)$.
	Let $\mu>0$.
	Then there is a cover of $\rReal^n$ by a cylindrical collection of $\polyfd[\format][\degree,M,1/\mu]$ prepared real cellular maps $\{\phi_j:\ext{\cell_j}{1/2}\to\cComplex^n\}$, each of them in $\FD*[\format][\degree]$ and compatible with the zero-sets of the coordinate functions $x_1,\dots,x_n$, such that for each $j$ we have the following expansion of each of the functions $f_i$ in $\phi_j(\rRealPos\cell_j)$:
	\begin{equation}
		\label{eq: prepared expansion}
		f_i(\varX)=\prod_{k=1}^n \abs{\varX_k-\theta_{j,k}(\initial{\varX}{1}{k-1})}^{\alpha_{i,j,k}}
		\cdot
		U_{i,j}(\varX),
	\end{equation}
	where $\alpha_{i,j,k}\in\qRationals$ is of size at most $\polyfd$ and the functions $\theta_{j,k},U_{i,j}:\phi_{j}(\rRealPos\cell_j)\to\rReal$ are in $\FD*[\format][\degree]$.
	The sign of $U_{i,j}\comp \phi_j$ is constant on $\rRealPos\cell_j$, and, if this sign is not $0$, we have that the diameter of  $\log \abs{U_{i,j}}\comp \phi_j (\rRealPos\cell_j)$ is less than~$\mu$.

	In addition, if $\theta_{j,k}$ is not identically $0$ over $\phi_j(\rRealPos\cell_j)$, then it is nowhere vanishing and we have 
	\begin{equation}
		\label{eq: center condition}
		\abs{\varX_k - \theta_{j,k}(\initial{\varX}{1}{k-1})}\leq \mu\abs{\varX_k}
	\end{equation}
	for all $\varX\in\phi_{j}(\rRealPos\cell_j)$.
	If, moreover, we have $\alpha_{i,j,k}\neq 0$ for some $i$, then we also have that the left-hand side of \eqref{eq: center condition} is nowhere vanishing over $\phi_j(\rRealPos\cell_j)$.
\end{theorem}
\begin{figure}
	\includegraphics[width=0.66\linewidth]{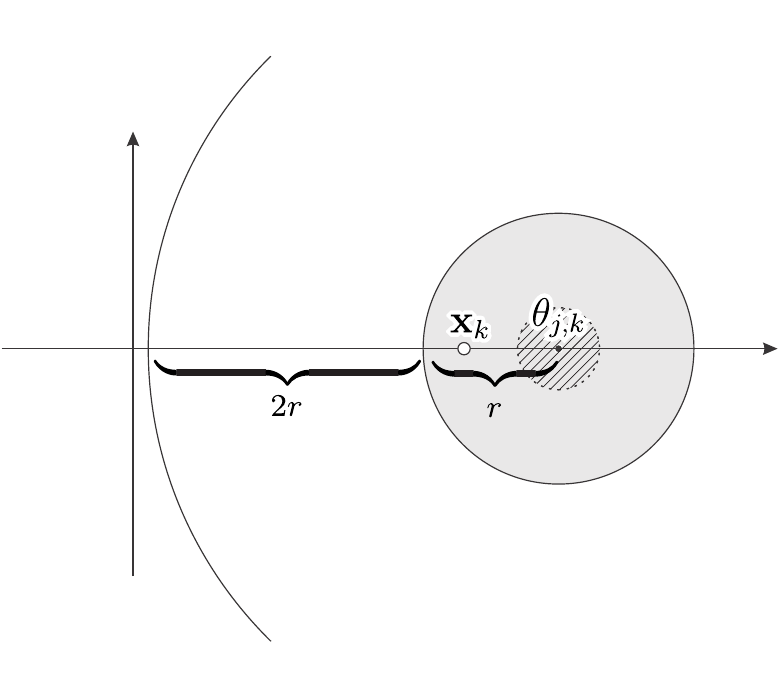}
	\caption{\emph{Condition on a cell with non-zero center.}
	Since the $1/3$-extension of the disc $\disc[r]+\theta_{j,_k}$ does not meet the origin, we have that $\abs{\varX_k-\theta_{j,k}}<r$ and $2r<\abs{\varX_k}$ for all $\varX_k\in\annulus[\cdot][r]+\theta_{j,_k}$.
	Hence $\abs{\varX_k-\theta_{j,k}}<\frac{1}{2}\abs{\varX_k}$.
	In general, for a prepared map $\varX_k=\varZ_k^{q_k}+\theta_{j,k}$ and a $\delta$-extension, we have $\abs{\varX_k-\theta_{j,k}}<\frac{\delta^{q_k}}{1-\delta^{q_k}}\abs{\varX_k}$.
	}
	\label{fig: center condition}
\end{figure}
\begin{proof}
	We apply \Cref{thm: real parameterization} to the graphs of $f_1,\dots,f_M$ and to the zero-sets of the coordinate functions $\varX_1,\dots\varX_{n+1}$, obtaining a cover of $\rReal^n$ by a cylindrical collection of prepared real cellular maps $\{\phi_j:\hExt*{\cell_j\odot\cell[F]_j}{\sigma}\to\cComplex^{n+1}\}$ of size $\polyfd[\format][\degree,M,1/\sigma]$, where the value of $\sigma$ will be chosen later.
	Write the $k$-th coordinate of $\phi_j$ as $\phi_{j,k}(\varZ)=\pm\varZ_k^{q_k}+\varphi_{j,k}(\initial{\varZ}{1}{k-1})$.
	Each of these coordinates is in $\FD*[\format][\degree]$.

	By choosing $\sigma=\polyfd^{-1}\cdot\mu$, the monomialization lemma (\Cref{lem: monomialization}) and our assumption on the compatibility of $\phi_j$ with the zero-sets of the coordinate functions implies that $\phi_{j,k}$ is either identically $0$ on $\hExt*{\cell_j\odot\cell[F]_j}{\sigma}$ or we may write $\phi_{j,k}(\varZ)=\initial{\varZ}{1}{k}^{\boldsymbol{\alpha}_{j,k}}\cdot \widetilde U_{j,k}(\initial{\varZ}{1}{k})$, where $\boldsymbol{\alpha}_{j,k}\in\zIntegers^k$ is as in the statement of \Cref{lem: monomialization} and $\log|\widetilde U_{j,k}|$ maps $\cell_j\odot\cell[F]_j$ to a set of diameter smaller than $\mu$.
	
	In particular, this is true for those $\phi_{j,n+1}$ which parameterize the graph of one of the functions $f_i$ over the base $(\initial{(\phi_{j})}{1}{n})(\rRealPos(\cell_j))$.
	For these cells, we must have $\cell[F]_j=\point$, and so $\phi_{j,n+1}(\varZ)=\initial{\varZ}{1}{n}^{\boldsymbol{\alpha}_{j,n+1}}\cdot \widetilde U_{j,n+1}(\initial{\varZ}{1}{n})$ depends only on the first $n$ coordinates $\initial{\varZ}{1}{n}$.
	In the image of such a cell, we have that $\varX_k=\phi_{j,k}$ and so
	\begin{equation}
		\label{eq: x to z coordinates}
		\abs{\varZ_k}=\abs{\varX_k-\varphi_{j,k}(\initial{\varZ}{1}{k-1})}^{1/q_k}.
	\end{equation}
	Letting $\theta_{j,k}$ and $U_{j,k}$ be obtained from $\varphi_{j,k}$ and $\widetilde U_{j,k}$, respectively, by recursively substituting the $\varX$ coordinates for $\varZ$ coordinates as in \eqref{eq: x to z coordinates}, we obtain expansions of the functions $f_i$ as in \eqref{eq: prepared expansion}.

	It remains to show that, if one of the centers $\theta_{j,k}$ is not identically $0$, then we have \eqref{eq: center condition}.
	This follows easily from the condition in the end of \Cref{rem: no bad annuli} (see \Cref{fig: center condition} for an example in the case of a translate map where $\hyperbolicParameter{\sigma}=1/3$).
	If $\alpha_{i,j,k}\neq 0$ for some $i$, then by \Cref{lem: monomialization} we have that the $k$-th coordinate of $\cell_j$ is of type $\puncDisc$, $\puncDisc*$ or $\annulus$.
	In particular, we have that $\varZ_k\neq 0$ and hence also~$\varX_k\neq \theta_{j,k}(\initial{\varX}{1}{k-1})$.
\end{proof}

\enlargethispage{1cm}
\bibliographystyle{abbrv}
\bibliography{references}

\end{document}